\documentclass[english]{article}
\usepackage[T1]{fontenc}
\usepackage[latin9]{inputenc}
\usepackage{babel}
\usepackage{amsthm}
\usepackage{amsmath}
\usepackage{amssymb}
\usepackage[authoryear]{natbib}
\usepackage{subscript}
\usepackage[unicode=true]
 {hyperref}
\usepackage{breakurl}

\makeatletter
  \theoremstyle{definition}
  \newtheorem{defn}{\protect\definitionname}
  \theoremstyle{plain}
  \newtheorem{lem}{\protect\lemmaname}
 \ifx\proof\undefined\
   \newenvironment{proof}[1][\proofname]{\par
     \normalfont\topsep6\p@\@plus6\p@\relax
     \trivlist
     \itemindent\parindent
     \item[\hskip\labelsep
           \scshape
       #1]\ignorespaces
   }{%
     \endtrivlist\@endpefalse
   }
   \providecommand{\proofname}{Proof}
 \fi
  \theoremstyle{plain}
  \newtheorem{cor}{\protect\corollaryname}
  \theoremstyle{plain}
  \newtheorem{prop}{\protect\propositionname}

\makeatother

\providecommand{\corollaryname}{Corollary}
\providecommand{\definitionname}{Definition}
\providecommand{\lemmaname}{Lemma}
\providecommand{\propositionname}{Proposition}

\begin{document}

\title{Coherent prior distributions in univariate finite mixture and Markov-switching
models}

\author{\L{}ukasz Kwiatkowski%
\thanks{Department of Econometrics and Operations Research, Cracow University
of Economics, Cracow, Poland, \protect\href{mailto:kwiatkol@uek.krakow.pl}{kwiatkol@uek.krakow.pl}%
}}
\maketitle
\begin{abstract}
Finite mixture and Markov-switching models generalize and, therefore,
nest specifications featuring only one component. While specifying
priors in the two: the general (mixture) model and its special (single-component)
case, it may be desirable to ensure that the prior assumptions introduced
into both structures are coherent in the sense that the prior distribution
in the nested model amounts to the conditional prior in the mixture
model under relevant parametric restriction. The study provides the
rudiments of setting coherent priors in Bayesian univariate finite
mixture and Markov-switching models. Once some primary results are
delivered, we derive specific conditions for coherence in the case
of three types of continuous priors commonly engaged in Bayesian modeling:
the normal, inverse gamma, and gamma distributions. Further, we study
the consequences of introducing additional constraints into the mixture
model's prior (such as the ones enforcing identifiability or some
sort of regularity, e.g. second-order stationarity) on the coherence
conditions. Finally, the methodology is illustrated through a discussion
of setting coherent priors for a class of Markov-switching AR(2) models.
\end{abstract}

\section{Introduction\label{sec:Introduction}}

Consider two statistical models, say, \emph{M\textsubscript{G}} and
\emph{M\textsubscript{R},} such that the latter constitutes a special
case of the former under some parametric restriction, and let vectors
$\theta^{(G)}$ and $\theta^{(R)}$ collect their parameters, respectively.
Note that $\theta^{(G)}$ includes $\theta^{(R)}$, which thereby
is the vector of common parameters (as opposed to the vector of\emph{
M\textsubscript{G}}'s specific coefficients, say, $\gamma$, so that
$\theta^{(G)}=(\theta^{(R)\prime}\;\gamma^{\prime})^{\prime}$, $A^{\prime}$
symbolizing the transpose of any matrix \emph{A}). Let $\gamma_{0}$
be the value of $\gamma$ under which \emph{M\textsubscript{G}} collapses
to \emph{M\textsubscript{R}}. In what follows, we adopt notational
convention under which, generally, $\pi_{\underline{\omega}}(\omega|M)$
denotes the p.d.f. of some random variable $\underline{\omega}$ at
$\underline{\omega}=\omega$ under model \emph{M}. Analogously, $\pi_{\underline{\omega}|\underline{\gamma}}(\omega|\gamma,M)$
stands for the p.d.f. of $\underline{\omega}$'s conditional distribution
at $\underline{\omega}=\omega$ given $\underline{\gamma}=\gamma$.
Finally, to avoid measure-theoretic intricacies, though with some
abuse of notation, we use the above symbols of density functions to
refer to the underlying distributions as well.

Within a non-Bayesian statistical framework, that \emph{M\textsubscript{G}}
nests \emph{M\textsubscript{R} }amounts to the equality of corresponding
sample distributions under the nesting constraint, i.e., $\pi_{\underline{y}|\underline{\theta^{(R)}}}(y|\theta^{(R)},M_{R})=\pi_{\underline{y}|\underline{\theta^{(R)}},\underline{\gamma}}(y|\theta^{(R)},\gamma=\gamma_{0},M_{G})$
for any $y\in Y\subseteq\mathbb{R}^{T}$. However, should the models
in question be regarded Bayesian, then nesting \emph{M\textsubscript{R}
}in \emph{M\textsubscript{G}} would also require that the prior information
introduced in the former be ``nested'' in the one incorporated into
the general structure. It follows then that $\pi_{\underline{\theta^{(R)}}}(\theta^{(R)}|M_{R})$
should be induced from $\pi_{\underline{\theta^{(G)}}}(\theta^{(G)}|M_{G})$
via conditioning upon the reducing restriction. The definition below
formalizes the concept of such prior coherence.
\begin{defn}
[Prior coherence]\emph{If the prior distributions: $\pi_{\underline{\theta^{(R)}}}(\theta^{(R)}|M_{R})$
and $\pi_{\underline{\theta^{(G)}}}(\theta^{(G)}|M_{G})$, satisfy
the condition:\label{Definition1}}
\begin{equation}
\pi_{\underline{\theta^{(R)}}}(\theta^{(R)}|M_{R})=\pi_{\underline{\theta^{(R)}}|\underline{\gamma}}(\theta^{(R)}|\gamma=\gamma_{0},M_{G}),\label{eq:1}
\end{equation}
\emph{then they are called coherent, and the models $M_{G}$ and $M_{R}$
are said to feature coherent prior structures.}
\end{defn}
Note that if $\theta^{(R)}$ and $\gamma$ in the $M_{G}$ model are
\emph{a priori} independent, then it is required for the prior coherence
that the prior of $\theta^{(R)}$ be the same in both models, i.e.,
$\pi_{\underline{\theta^{(R)}}}(\theta^{(R)}|M_{R})=\pi_{\underline{\theta^{(R)}}}(\theta^{(R)}|M_{G})$. 

The idea of specifying coherent prior distributions has been originated
by \citet{Dickey:1974:PriorCoherence} and \citet{Poirier:1985:PriorCoherence}
in the context of hypothesis testing within linear models. We refer
the reader to \citet{ConsonniVeronese:2008:PriorCompatibility} for
a recent study and literature review on various forms of prior compatibility
across linear models.

Obviously, the idea of establishing coherent prior structures over
various models does not pertain to the class of the linear specifications
solely, but applies whenever the nesting comes into play. In particular,
the mixture (and Markov-switching) models nest their single-component
counterparts, the latter being derived from the former via relevant
equality restrictions. Perversely, one may argue, however, that there
is no compelling reason within the subjective framework to relate
priors across models, since they express subjective opinions conditionally
on a different state of information. Nevertheless, ensuring prior
coherence across various model specifications appears crucial to the
model comparison (usually performed via recognizably prior-sensitive
Bayes factors), for reconciling the models' prior structures sheds
some layer of arbitrariness (\citealt{DawidLauritzen:2001:PriorCompatibility},
\citealt{ConsonniVeronese:2008:PriorCompatibility}). In particular,
within the finite mixture and Markov-switching class of models, specifying
coherent priors may be desirable for testing the relevance of incorporating
the mixture (switching) structure into the otherwise single-component
specification. To the author's best knowledge, the issue of prior
compatibility within the mixture models has not been raised in the
literature so far. Therefore, in the present research, we take an
interest in settling coherent prior structures for the mixture models
(and the Markov-switching structures alike) and their single-component
counterparts. 

In Section \ref{sec:Prior-coherence-in} we lay the basic foundations
of establishing coherent prior structures within the finite mixture
and the Markov-switching model frameworks, and arrive at the basic
lemma. The results incline us to focus next on exponential families
of prior distributions, for three representatives of which, namely
the normal, inverse gamma, and gamma distributions, we derive in \ref{sec:Specific-results}
explicit conditions relating the hyperparameters of the general and
the nested model. Section \ref{sec:Coherence-of-constrained} is devoted
to the cases in which the priors are subject to certain restrictions,
such as the ones enforcing identifiability of the mixture components
(via an inequality constraint imposed on a group of mixture parameters)
or some sort of regularity (e.g., the second-order stationarity).
Finally, in Section \ref{sec:Example}, the methodology is illustrated
with a discussion of setting coherent priors for a class of Markov-switching
AR(2) models.

\section{Prior coherence in the mixture and Markov-switching models\label{sec:Prior-coherence-in}}

Consider a single-component model, $M_{1}$, with parameters collected
in
\[
\theta^{(1)}=(\delta^{\prime}\;\lambda_{1,1}\;\lambda_{1,2}\;\ldots\;\lambda_{1,n})^{\prime}\in\Theta^{(1)},\qquad n\in\mathbb{N},
\]
and the general, \emph{K}-component mixture model, $M_{K}$, $K\in\mathbb{N}$,
with parameters
\[
\theta^{(K)}=(\delta^{\prime}\;\lambda_{1}^{(K)\prime}\;\lambda_{2}^{(K)\prime}\;\ldots\;\lambda_{n}^{(K)\prime}\;\eta^{\prime})^{\prime}\in\Theta^{(K)}.
\]
The following remarks clarify our notational convetion:
\begin{itemize}
\item The vector $\delta$ is comprised of the parameters that are non-mixture
and common to both models.
\item The parameters $\lambda_{1,j}$ $(j=1,2,\ldots,n)$ in the model $M_{1}$
are scalar, with the first subscript indicating that the model features
a single component.
\item Each vector $\lambda_{j}^{(K)}$ $(j=1,2,\ldots,n)$ in the model
$M_{K}$ collects \emph{K} parameters that arise as a result of introducing
the \emph{K}-component mixture structure into the corresponding (scalar)
parameter $\lambda_{1,j}$ in $M_{1}$, so that $\lambda_{j}^{(K)}=(\lambda_{1,j}\;\lambda_{2,j}\;\ldots\;\lambda_{K,j})^{\prime}$.
Note that the first coordinate in $\lambda_{j}^{(K)}$, denoted by
$\lambda_{1,j}$, coincides with the corresponding parameter in the
single-component model.
\item The vector $\eta=(\eta_{1}\;\eta_{2}\;\ldots\;\eta_{K})^{\prime}$
in the model $M_{K}$ contains the probability parameters:

\begin{itemize}
\item If $M_{K}$ is a finite mixture model, then $\eta_{i}$ $(i=1,2,\ldots,K)$
are the mixture probabilities, and $\eta\in\Delta^{(K-1)}$, where
$\Delta^{(K-1)}$ denotes the unit (\emph{K}-1)-simplex.
\item If $M_{K}$ is a Markov-switching model, with $\{S_{t};\; t=0,1,...\}$
forming the underlying \emph{K}-state (homogenous) Markov chain, then
$\eta_{i}$ $(i=1,2,\ldots,K)$ are the rows of transition matrix
$P=[\eta_{ij}]_{i,j=1,2,\ldots,K}$, $\eta_{ij}\equiv\Pr{(S_{t}=j|S_{t-1}=i)}$,
i.e., $\eta_{i}=(\eta_{i1}\;\eta_{i2}\;\ldots\;\eta_{iK})\in\Delta^{(K-1)}$,
and therefore $\eta\in(\Delta^{(K-1)})^{K}$. For simplicity, though
without loss of generality, we assume that the chain's initial state
distribution: $\xi=(\xi_{1}\;\xi_{2}\;\ldots\;\xi_{K})^{\prime}$,
$\xi_{i}\equiv\Pr{(S_{0}=i)}$, $i=1,2,\ldots,K$, is either known
(e.g., a uniform distribution) or equal to the chain's ergodic distribution
(which introduces into the probabilities $\xi_{i}$'s conditioning
upon the transition matrix, $\xi_{i}\equiv\Pr{(S_{0}=i|P)}$).
\end{itemize}
\end{itemize}
Notice that the two: $M_{1}$ and $M_{K}$, represent the extremes,
i.e, at one end, there is the single-component model $M_{1},$ whereas
at the other - the specification $M_{K}$, in which all $\lambda_{j}^{(K)}$'s
constitute the mixture counterparts of $\lambda_{1,j}$'s in $M_{1}$.
Obviously, there are $2^{n}-2$ specifications in between, such that
only some of $\lambda_{j}^{(K)}$'s are actually the vectors of mixture
parameters, whereas the other ones remain equivalent to the corresponding
coefficients in the single-component model. These ``intermediate''
model structures encompass $M_{1}$ on the one hand, and, on the other,
are nested within the most general one, i.e., $M_{K}$. Nevertheless,
we limit most of our further considerations only to the two extreme
cases, for the reason that, under the assumptions of our analysis,
establishing coherent priors for the two: the single-component model
and any of the ``intermediate'' constructions, comes down to the
same framework by means of relegating those $\lambda_{1,j}$'s that
are non-mixture in both specifications to the vector of the common
parameters, $\delta$. In a similar fashion, coherence of the ``intermediate''
and the general model can be settled, which would require including
also the probabilities $\eta$ in the common parameters vector. We
revisit the issue in the final paragraph of Section \ref{sec:Example}.

In what follows, for both models in question, prior independence is
assumed between the parameter vector's components:
\begin{equation}
\pi_{\underline{\theta^{(1)}}}(\theta^{(1)}|M_{1})=\pi_{\underline{\delta}}(\delta|M_{1})\overset{n}{\underset{j=1}{\prod}}\pi_{\underline{\lambda_{1,j}}}(\lambda_{1,j}|M_{1}),\label{eq:2}
\end{equation}
\begin{equation}
\pi_{\underline{\theta^{(K)}}}(\theta^{(K)}|M_{K})=\pi_{\underline{\delta}}(\delta|M_{K})\left[\overset{n}{\underset{j=1}{\prod}}\pi_{\underline{\lambda_{j}^{(K)}}}(\lambda_{j}^{(K)}|M_{K})\right]\pi_{\underline{\eta}}(\eta|M_{K}).\label{eq:3}
\end{equation}
Moreover, for each $j=1,2,\ldots,n$, also the individual coordinates
of $\lambda_{j}^{(K)}$ are presumed \emph{a priori }independent:
\begin{equation}
\pi_{\underline{\lambda_{j}^{(K)}}}(\lambda_{j}^{(K)}|M_{K})=\overset{K}{\underset{i=1}{\prod}}\pi_{\underline{\lambda_{i,j}}}(\lambda_{i,j}|M_{K}).\label{eq:4}
\end{equation}
Finally, all the\textbf{ }priors under consideration are assumed proper,
for it may be shown that setting improper priors in mixture models
yields improper posteriors (see \citealt{RoederWasserman:1997:MixtureOfNormals},
\citealt{FruhwirthSchnatter:2006MixtureAndMS}). 

Resting upon (\ref{eq:2}) and (\ref{eq:3}), the priors: $\pi_{\underline{\theta^{(1)}}}(\theta^{(1)}|M_{1})$
and $\pi_{\underline{\theta^{(K)}}}(\theta^{(K)}|M_{K})$, are coherent
if the two conditions are met simultaneously:
\begin{equation}
\pi_{\underline{\delta}}(\delta|M_{1})=\pi_{\underline{\delta}}(\delta|M_{K})\label{eq:5}
\end{equation}
and, for all $j=1,2,\ldots,n$,
\begin{equation}
\pi_{\underline{\lambda_{1,j}}}(\lambda_{1,j}|M_{1})\propto\pi_{\underline{\lambda_{j}^{(K)}}}(\lambda_{j}^{(K)}|\lambda_{1,j}\equiv\lambda_{2,j}=\lambda_{3,j}=\ldots=\lambda_{K,j},M_{K}).\label{eq:6}
\end{equation}
Note that the postulated conditions do not explicitly concern the
mixture probabilities $(\eta)$, as these are either entirely absent
from the reduced model $(M_{1})$ or contained in the vector $\delta$
(in the case of establishing coherent prior structures for the general
and some ``intermediate'' model specification; then, (\ref{eq:5})
applies).

In order to rewrite (\ref{eq:6}) in terms of Definition \ref{Definition1},
the general model needs to be suitably reparametrized. Let $\widetilde{M}_{K}$
be the reparametrized model, with the parameters grouped in
\[
\widetilde{\theta}^{(K)}=(\delta^{\prime}\;\widetilde{\lambda}_{1}^{(K)\prime}\;\widetilde{\lambda}_{2}^{(K)\prime}\;\ldots\;\widetilde{\lambda}_{n}^{(K)\prime}\;\eta^{\prime})^{\prime}\in\widetilde{\Theta}^{(K)}.
\]
Each $\widetilde{\lambda}_{j}^{(K)}$ $(j=1,2,\ldots,n)$ is obtained
from the corresponding $\lambda_{j}^{(K)}$ via a transformation $g:\mathbb{{R}}^{K}\rightarrow\mathbb{{R}}^{K}$:
\[
\widetilde{\lambda}_{j}^{(K)}=g(\lambda_{j}^{(K)})=\begin{pmatrix}g_{1}(\lambda_{1,j})\\
g_{2}(\lambda_{2,j})\\
g_{3}(\lambda_{3,j})\\
\vdots\\
g_{K}(\lambda_{K,j})
\end{pmatrix}=\begin{pmatrix}\lambda_{1,j}\\
\lambda_{2,j}-\lambda_{1,j}\\
\lambda_{3,j}-\lambda_{1,j}\\
\vdots\\
\lambda_{K,j}-\lambda_{1,j}
\end{pmatrix}\equiv\begin{pmatrix}\lambda_{1,j}\\
\tau_{j}
\end{pmatrix},
\]
with $\tau_{j}=(\tau_{2,j}\;\tau_{3,j}\;\ldots\;\tau_{K,j})^{\prime}$
collecting the contrasts $\tau_{i,j}=\lambda_{i,j}-\lambda_{1,j}$
$(i=2,3,...,K)$. The inverse transformation follows as
\[
g^{-1}(\widetilde{\lambda}_{j}^{(K)})=\begin{pmatrix}g_{1}^{-1}(\lambda_{1,j})\\
g_{2}^{-1}(\lambda_{2,j})\\
g_{3}^{-1}(\lambda_{3,j})\\
\vdots\\
g_{K}^{-1}(\lambda_{K,j})
\end{pmatrix}=\begin{pmatrix}\lambda_{1,j}\\
\tau_{2,j}+\lambda_{1,j}\\
\tau_{3,j}+\lambda_{1,j}\\
\vdots\\
\tau_{K,j}+\lambda_{1,j}
\end{pmatrix}=\begin{pmatrix}\lambda_{1,j}\\
\lambda_{2,j}\\
\lambda_{3,j}\\
\vdots\\
\lambda_{K,j}
\end{pmatrix}\equiv\lambda_{j}^{(K)}.
\]
Owing to the fact that $\left|\frac{\partial g^{-1}(\widetilde{\lambda}_{j}^{(K)})}{\partial\widetilde{\lambda}_{j}^{(K)}}\right|=1$,
the p.d.f. of $\widetilde{\lambda}_{j}^{(K)}$'s prior can be easily
derived:
\begin{align}
\pi_{\underline{\widetilde{\lambda}_{j}^{(K)}}}(\widetilde{\lambda}_{j}^{(K)}|\widetilde{M}_{K}) & =\pi_{\underline{\lambda_{j}^{(K)}}}\left(g^{-1}(\widetilde{\lambda}_{j}^{(K)})|M_{K}\right)\nonumber \\
 & =\pi_{\underline{\lambda_{1,j}}}\left(g_{1}^{-1}(\lambda_{1,j})|M_{K}\right)\overset{K}{\underset{i=2}{\prod}}\pi_{\underline{\lambda_{i,j}}}\left(g_{i}^{-1}(\tau_{i,j})|M_{K}\right)\label{eq:7}\\
 & =\pi_{\underline{\lambda_{1,j}}}\left(\lambda_{1,j}|M_{K}\right)\overset{K}{\underset{i=2}{\prod}}\pi_{\underline{\lambda_{i,j}}}\left(\tau_{i,j}+\lambda_{1,j}|M_{K}\right).\nonumber 
\end{align}

Now, recall that $M_{1}$ results from $M_{K}$ under the equality
constraint of all the coordinates within each vector $\lambda_{j}^{(K)}$,
i.e.,
\begin{equation}
\lambda_{1,j}=\lambda_{2,j}=\ldots=\lambda_{K,j},\qquad j=1,2,\ldots n.\label{eq:8}
\end{equation}
 In the model $\widetilde{M}_{K}$, (\ref{eq:8}) is equivalent to
setting all the corresponding contrasts to zero:
\[
\lambda_{1,j}=\lambda_{2,j}=\ldots=\lambda_{K,j}\Leftrightarrow\tau_{2,j}=\tau_{3,j}=\ldots=\tau_{K,j}=0\Leftrightarrow\tau_{j}=0_{[(K-1)\times1]}.
\]
Conditions (\ref{eq:5}) and (\ref{eq:6}) can now be restated in
terms of the reparametrized model:
\begin{equation}
\pi_{\underline{\delta}}(\delta|M_{1})=\pi_{\underline{\delta}}(\delta|\widetilde{M}_{K}),\label{eq:9}
\end{equation}
and
\begin{equation}
\pi_{\underline{\lambda_{1,j}}}(\lambda_{1,j}|M_{1})=\pi_{\underline{\lambda_{1,j}}|\underline{\tau_{j}}}(\lambda_{1,j}|\tau_{j}=0_{[(K-1)\times1]},\widetilde{M}_{K})\label{eq:10}
\end{equation}
for each $j=1,2,\ldots,n$. Note that the prior distribution of $\delta$
in $\widetilde{M}_{K}$, appearing on the right-hand side of (\ref{eq:9}),
is actually equal to $\pi_{\underline{\delta}}(\delta|M_{K})$, for
the transformation \emph{g} does not affect the parameters collected
in $\delta$.

We end this section by formulating our basic result in Lemma 1, with
Corollary 1 following immediately.
\begin{lem}
For a given $j\in\{1,2,\ldots,n\}$, the prior distribution of $\lambda_{j}^{(K)}$
under $M_{K}$, and the one of the corresponding parameter $\lambda_{1,j}$
under $M_{1}$ are coherent iff\label{lem:Lemma1}
\begin{equation}
\pi_{\underline{\lambda_{1,j}}}(\lambda_{1,j}|M_{1})\propto\overset{K}{\underset{i=1}{\prod}}\pi_{\underline{\lambda_{i,j}}}(\lambda_{1,j}|M_{K}).\label{eq:Lemma1_1}
\end{equation}
\end{lem}
\begin{proof}
Employing (\ref{eq:7}) and (\ref{eq:10}), we proceed as follows:
\begin{align*}
\pi_{\underline{\lambda_{1,j}}}(\lambda_{1,j}|M_{1}) & =\pi_{\underline{\lambda_{1,j}}|\underline{\tau_{j}}}(\lambda_{1,j}|\tau_{j}=0_{[(K-1)\times1]},\widetilde{M}_{K})\\
 & \propto\pi_{\underline{\lambda_{1,j}},\underline{\tau_{j}}}(\lambda_{1,j},\tau_{j}=0_{[(K-1)\times1]}|\widetilde{M}_{K})\\
 & \propto\pi_{\underline{\lambda_{1,j}}}(\lambda_{1,j}|M_{K})\left[\overset{K}{\underset{i=2}{\prod}}\pi_{\underline{\lambda_{i,j}}}(\lambda_{1,j}|M_{K})\right]\\
 & =\overset{K}{\underset{i=1}{\prod}}\pi_{\underline{\lambda_{i,j}}}(\lambda_{1,j}|M_{K}).
\end{align*}
\end{proof}
\begin{cor}
For a given $j\in\{1,2,\ldots,n\}$, if all densities $\pi_{\underline{\lambda_{i,j}}}(\cdot|M_{K})$
($i=1,2,\ldots,K$) are the same, i.e., $\pi_{\underline{\lambda_{i,j}}}(x|M_{K})=\pi_{\underline{\lambda_{1,j}}}(x|M_{K})$,
$x\in\mathbb{R}$, then (\ref{eq:Lemma1_1}) reduces to\label{cor:Cor1}
\begin{equation}
\pi_{\underline{\lambda_{1,j}}}(\lambda_{1,j}|M_{1})\propto\left[\pi_{\underline{\lambda_{1,j}}}(\lambda_{1,j}|M_{K})\right]^{K}.\label{eq:Lemma1_2}
\end{equation}

\end{cor}

\section{Specific results\label{sec:Specific-results}}

The relations presented in (\ref{eq:Lemma1_1}) and (\ref{eq:Lemma1_2})
may prompt one, quite instinctively, to consider some exponential
family for specyfing the prior densities of $\lambda_{i,j}$'s under
$M_{K}$, since such an approach would yield the same type of the
prior distribution for $\lambda_{1,j}$ under $M_{1}$. What remains
then is to determine the relationships between the hyperparameters
of all the relevant densities (belonging to a given exponential family).

In the subsections below we focus our attention on three exponential
families: the normal, inverse gamma, and gamma distributions, which,
for their property of (conditional) conjugacy, are commonly entertained
in Bayesian statistical modeling. In each case, we apply Lemma 1 and
Corollary 1 to derive explicit formulae relating the hyperparameters
of the general and the nested model. Throughout the section we fix
the index $j\in\{1,2,\ldots,n\}$, and, for the sake of transparency,
drop it from the notation (e.g., writing $\lambda_{i}$ instead of
$\lambda_{i,j}$).

\subsection{Normal priors\label{sub:Normal}}

The following proposition establishes the coherence conditions upon
the normality of $\lambda_{i}$'s in the mixture model.
\begin{prop}
Suppose that each $\lambda_{i}$ $(i=1,2,\ldots,K)$ under $M_{K}$
follows a univariate normal distribution with mean $m_{i}^{(K)}$
and variance $v_{i}^{(K)}$:\label{prop:Prop1}
\[
\pi_{\underline{\lambda_{i}}}(\lambda_{i}|M_{K})=f_{N}(\lambda_{i}|m_{i}^{(K)},v_{i}^{(K)}).
\]
Then, the coherent prior for $\lambda_{1}$ under $M_{1}$ is the
normal distribution with mean $m^{(1)}$ and variance $v^{(1)}$:
\[
\pi_{\underline{\lambda_{1}}}(\lambda_{1}|M_{1})=f_{N}(\lambda_{1}|m^{(1)},v^{(1)}),
\]
where
\begin{equation}
m^{(1)}=\frac{\overset{K}{\underset{i=1}{\sum}}\frac{m_{i}^{(K)}}{v_{i}^{(K)}}}{\overset{K}{\underset{i=1}{\sum}}\frac{1}{v_{i}^{(K)}}}\label{eq:Prop1_m-variance}
\end{equation}
and
\begin{equation}
v^{(1)}=\left(\overset{K}{\underset{i=1}{\sum}}\frac{1}{v_{i}^{(K)}}\right)^{-1}.\label{eq:Prop1_v-variance}
\end{equation}
Alternatively, under precision-parametrized normal densities, if
\[
\pi_{\underline{\lambda_{i}}}(\lambda_{i}|M_{K})=f_{N}\left(\lambda_{i}|m_{i}^{(K)},(\breve{v}_{i}^{(K)})^{-1}\right),\qquad\breve{v}_{i}^{(K)}\equiv(v_{i}^{(K)})^{-1},
\]
for each $i=1,2,\ldots,K$, then
\[
\pi_{\underline{\lambda_{1}}}(\lambda_{1}|M_{1})=f_{N}\left(\lambda_{1}|m^{(1)},(\breve{v}^{(1)})^{-1}\right),
\]
where
\begin{equation}
m^{(1)}=\frac{\overset{K}{\underset{i=1}{\sum}}\breve{v}_{i}^{(K)}m_{i}^{(K)}}{\overset{K}{\underset{i=1}{\sum}}\breve{v}_{i}^{(K)}}\label{eq:Prop1_m-precision}
\end{equation}
and
\begin{equation}
\breve{v}^{(1)}=\overset{K}{\underset{i=1}{\sum}}\breve{v}_{i}^{(K)}.\label{eq:Prop1_v-precision}
\end{equation}
\end{prop}
\begin{proof}
See Appendix A.
\end{proof}
Following immediately from Proposition \ref{prop:Prop1}, the corollary
below provides expressions for $m^{(1)}$, $v^{(1)}$ and $\breve{v}^{(1)}$
upon the component-wise equality of hyperparameters under $M_{K}$.
\begin{cor}
i) If $m_{1}^{(K)}=m_{2}^{(K)}=\ldots=m_{K}^{(K)}\equiv m^{(K)}$,
then\label{cor:Cor2}
\begin{equation}
m^{(1)}=m^{(K)}.\label{eq:Cor2i_m}
\end{equation}

ii) If $v_{1}^{(K)}=v_{2}^{(K)}=\ldots=v_{K}^{(K)}\equiv v^{(K)}$
(or, equivalently, $\breve{v}_{1}^{(K)}=\breve{v}_{2}^{(K)}=\ldots=\breve{v}_{K}^{(K)}\equiv\breve{v}^{(K)}$),
then 
\begin{equation}
m^{(1)}=\frac{1}{K}\sum_{i=1}^{K}m_{i}^{(K)},\label{eq:Cor2ii_m}
\end{equation}
\begin{equation}
v^{(1)}=\frac{1}{K}v^{(K)}\label{eq:Cor2ii_v}
\end{equation}
and
\begin{equation}
\breve{v}^{(1)}=K\breve{v}^{(K)}.\label{eq:Cor2ii_v-precision}
\end{equation}

\end{cor}
According to (\ref{eq:Prop1_m-variance}) and (\ref{eq:Prop1_m-precision}),
the mean of the coherent (normal) prior of $\lambda_{1}$ under the
nested model constitutes a weighted sum of the corresponding means
in the mixture model:
\[
m^{(1)}=\overset{K}{\underset{i=1}{\sum}}w_{i}m_{i}^{(K)},
\]
with the weights given by
\[
w_{i}=\frac{(v_{i}^{(K)})^{-1}}{\overset{K}{\underset{k=1}{\sum}}(v_{k}^{(K)})^{-1}}=\frac{\breve{v}_{i}^{(K)}}{\overset{K}{\underset{k=1}{\sum}}\breve{v}_{k}^{(K)}},\qquad i=1,2,\ldots K.
\]
The result collapses either to a simple average of the means (under
equal variances $v_{i}^{(K)}$; see (\ref{eq:Cor2ii_m})), or, eventually,
to the very mean $m^{(K)}$, should the means coincide in all the\textbf{
}priors $\pi_{\underline{\lambda_{i}}}(\lambda_{i}|M_{K})$, $i=1,2,\ldots,K$;
see (\ref{eq:Cor2i_m}).

As regards the relationship between the dispersion of the priors,
from (\ref{eq:Prop1_v-variance}) it follows that the variance $v^{(1)}$
in the coherent prior of $\lambda_{1}$ under $M_{1}$ amounts to
a \emph{K}-th of the harmonic mean of the individual variances $v_{i}^{(K)}$,
$i=1,2,\ldots,K$. The result immediately translates to the relation
between the corresponding precisions, in terms of which $\breve{v}^{(1)}$
in the reduced model should be the sum of the precisions specified
in the general construction; see (\ref{eq:Prop1_v-precision}). Under
the special case of equal prior variances of all $\lambda_{i}$'s
in $M_{K}$, the resulting variance of $\lambda_{1}$ in $M_{1}$
reduces to a \emph{K}-th of the one assumed within the mixture model;
see (\ref{eq:Cor2ii_v}). Equivalently, the precision $\breve{v}^{(1)}$
is \emph{K} times the one predetermined for $\lambda_{i}$'s, thereby
growing proportionally to the number of the mixture components; see
(\ref{eq:Cor2ii_v-precision}).

We end this subsection by noticing that under the assumptions of Proposition
\ref{prop:Prop1} it is only possible to determine the hyperparameters
in the single-component specification, based on the ones prespecified
in the mixture model, and not the reverse. However, adopting an additional
assumption of the equal prior means and, simultaneously, variances
of $\lambda_{i}$'s under $M_{K}$, allows one to predetermine the
hyperparameters for $\lambda_{1}$ in the nested model first (i.e.,
$m^{(1)}$ and $v^{(1)}$), and then the ones in the general specification
(i.e., $m^{(K)}$ and $v^{(K)}$), employing (\ref{eq:Cor2i_m}) and
(\ref{eq:Cor2ii_v}) (or, equivalently, (\ref{eq:Cor2ii_v-precision})).
The latter idea appears to gain particular importance while considering
models with various number of the mixture components: $M_{K}$ with
$K\in\{K_{min},K_{min}+1,\ldots,K_{max}\}=\mathbb{{K}}$, $K_{min}\ge2$,
along the single-component structure, $M_{1}$. Since the latter constitutes
a special case of all the mixture models under consideration, one
may naturally be prompted to set the hyperparameters under $M_{1}$
first, and then invoke (\ref{eq:Cor2i_m}) and (\ref{eq:Cor2ii_v})
(or, (\ref{eq:Cor2ii_v-precision})) to calculate coherent values
of $m^{(K)}$ and $v^{(K)}$ (or, $\breve{v}^{(K)}$) for each $K\in\mathbb{{K}}$.
Intuitively, though not in the sense of Definition \ref{Definition1},
such an approach would endow the priors of all the\textbf{ }models
with some sort of compatibility, by means of ensuring prior coherence
of the single-component model with each of the mixture specifications
individually.

\subsection{Inverse gamma priors\label{sub:Inverse-gamma}}

We move on to deriving the coherence conditions under setting inverse
gamma priors for $\lambda_{i}$'s in the mixture model.
\begin{prop}
Suppose that each $\lambda_{i}$ $(i=1,2,\ldots,K)$ under $M_{K}$
follows an inverse gamma distribution with shape parameter $a_{i}^{(K)}>0$
and scale parameter $b_{i}^{(K)}>0$:\label{prop:Prop2}
\begin{align*}
\pi_{\underline{\lambda_{i}}}(\lambda_{i}|M_{K}) & =f_{IG}(\lambda_{i}|a_{i}^{(K)},b_{i}^{(K)})\\
 & =\frac{1}{(b_{i}^{(K)})^{a_{i}^{(K)}}\Gamma(a_{i}^{(K)})}(\lambda_{i})^{-(a_{i}^{(K)}+1)}\exp\left\{ -\frac{1}{b_{i}^{(K)}\lambda_{i}}\right\} .
\end{align*}
Then, the coherent prior for $\lambda_{1}$ under $M_{1}$ is the
inverse gamma distribution with shape parameter $a^{(1)}$ and scale
parameter $b^{(1)}$:
\[
\pi_{\underline{\lambda_{1}}}(\lambda_{1}|M_{1})=f_{IG}(\lambda_{1}|a^{(1)},b^{(1)}),
\]
where
\begin{equation}
a^{(1)}=\overset{K}{\underset{i=1}{\sum}}a_{i}^{(K)}+K-1\label{eq:Prop2_a}
\end{equation}
and
\begin{equation}
b^{(1)}=\left(\overset{K}{\underset{i=1}{\sum}}\frac{1}{b_{i}^{(K)}}\right)^{-1}.\label{eq:Prop2_b}
\end{equation}
\end{prop}
\begin{proof}
See Appendix B.
\end{proof}
The formulae for $a^{(1)}$ and $b^{(1)}$ in the special cases of
component-wise equal hyperparameters under the general model follow
directly from Proposition \ref{prop:Prop2} and are stated in the
corollary below.
\begin{cor}
i) If $a_{1}^{(K)}=a_{2}^{(K)}=\ldots=a_{K}^{(K)}\equiv a^{(K)}$,
then\label{cor:Cor3}
\begin{equation}
a^{(1)}=Ka^{(K)}+K-1.\label{eq:Cor3i_a}
\end{equation}

ii) If $b_{1}^{(K)}=b_{2}^{(K)}=\ldots=b_{K}^{(K)}\equiv b^{(K)}$,
then
\begin{equation}
b^{(1)}=\frac{1}{K}b^{(K)}.\label{eq:Cor3ii_b}
\end{equation}

\end{cor}
The relationship between the shape parameters, given by (\ref{eq:Prop2_a}),
suggests that $a^{(1)}$ is an increasing function of the number of
the mixture components (partly on account of its formula involving
the sum of $a_{i}^{(K)}$'s), whereas the scale parameters, $b_{i}^{(K)}$'s
and $b^{(1)}$, are interrelated in the same fashion as the variances
in the case of the normal priors, examined in the previous subsection
(see (\ref{eq:Prop2_b}) and (\ref{eq:Prop1_v-variance})).

Similarly to the previous one, Proposition \ref{prop:Prop2} enables
one to derive coherent values of the hyperparameters under $M_{1},$
based on the ones prespecified under the mixture model, unless these
are held equal across the mixture components (see Corollary \ref{cor:Cor3}).
Turning to the special case of $a_{1}^{(K)}=a_{2}^{(K)}=\ldots=a_{K}^{(K)}\equiv a^{(K)}$,
let us transform (\ref{eq:Cor3i_a}) into
\begin{equation}
a^{(K)}=\frac{a^{(1)}-K+1}{K},\label{eq:25}
\end{equation}
which would be of use once we were to establish the coherent prior
in $M_{K}$, based on the predetermined value of the relevant hyperparameter
in $M_{1}$. Interestingly, to guarantee the positivity of $a^{(K)}$
(as a shape parameter of an inverse gamma distribution) it requires
that
\begin{equation}
a^{(1)}>K-1,\label{eq:Gam_restr_a1}
\end{equation}
which explicitly takes the number of mixture components into account.
Now, evoke the context of handling models $M_{K}$ with various $K\in\mathbb{{K}}$,
as outlined at the end of the previous subsection. In order to ascertain
the prior under each of them coherently with the one prespecified
for the single-component model, the condition
\begin{equation}
a^{(1)}>K_{max}-1\label{eq:Gam_restr_a1_Kmax}
\end{equation}
must be satisfied. As long as (\ref{eq:Gam_restr_a1_Kmax}) holds,
the hyperparameters $a^{(K)}$ calculated through (\ref{eq:25}) are
positive for all $K\in\mathbb{{K}}$. Taking these remarks into account,
it emerges that once models with a different number of the components
are under consideration, it is crucial to fix \emph{a priori }its
maximum, $K_{max}$. With that provided, one proceeds to setting $a^{(1)}$
in compliance with (\ref{eq:Gam_restr_a1_Kmax}), and then to determining
$a^{(K)}$ via (\ref{eq:25}) for each $K\in\mathbb{{K}}$. Incidentally,
note that the issue pertains only to the shape parameters, while reconciling
the scale parameters: $b^{(1)}$ and $b^{(K)}$ (under $b_{1}^{(K)}=b_{2}^{(K)}=\ldots=b_{K}^{(K)}\equiv b^{(K)}$)
for each $K\in\mathbb{{K}}$, does not give rise to similar concerns.

\subsection{Gamma priors\label{sub:Gamma}}

Generally speaking, in some applications it is preferred to employ
the gamma distribution (rather than its inverse alternative) to specify
the prior. Therefore, we devote the present subsection to provide
the coherence conditions also in the case gamma priors are assumed
for all $\lambda_{i}$'s in the mixture model.
\begin{prop}
Suppose that each $\lambda_{i}$ $(i=1,2,\ldots,K)$ under $M_{K}$
follows a gamma distribution with shape parameter $\breve{a}_{i}^{(K)}>0$
and scale parameter $\breve{b}_{i}^{(K)}>0$:\label{prop:Prop3}
\begin{align*}
\pi_{\underline{\lambda_{i}}}(\lambda_{i}|M_{K}) & =f_{G}(\lambda_{i}|\breve{a}_{i}^{(K)},\breve{b}_{i}^{(K)})\\
 & =\frac{(\breve{b}_{i}^{(K)})^{\breve{a}_{i}^{(K)}}}{\Gamma(\breve{a}_{i}^{(K)})}(\lambda_{i})^{\breve{a}_{i}^{(K)}-1}\exp\left\{ -\breve{b}_{i}^{(K)}\lambda_{i}\right\} .
\end{align*}
Then, the coherent prior for $\lambda_{1}$ under $M_{1}$ is the
gamma distribution with shape parameter $\breve{a}^{(1)}$ and scale
parameter $\breve{b}^{(1)}$:
\[
\pi_{\underline{\lambda_{1}}}(\lambda_{1}|M_{1})=f_{G}(\lambda_{1}|\breve{a}^{(1)},\breve{b}^{(1)}),
\]
where
\begin{equation}
\breve{a}^{(1)}=\overset{K}{\underset{i=1}{\sum}}\breve{a}_{i}^{(K)}-K+1\label{eq:Prop3_a}
\end{equation}
and
\begin{equation}
\breve{b}^{(1)}=\overset{K}{\underset{i=1}{\sum}}\breve{b}_{i}^{(K)}.\label{eq:Prop3_b}
\end{equation}
\end{prop}
\begin{proof}
See Appendix C.
\end{proof}
Similarly as in the previous subsections, and following directly from
Proposition \ref{prop:Prop3}, the corollary below delivers expressions
for $\breve{a}^{(1)}$ and $\breve{b}^{(1)}$ under the special cases
of component-wise equal hyperparameters in the mixture model.
\begin{cor}
i) If $\breve{a}_{1}^{(K)}=\breve{a}_{2}^{(K)}=\ldots=\breve{a}_{K}^{(K)}\equiv\breve{a}^{(K)}$,
then\label{cor:Cor4}
\begin{equation}
\breve{a}^{(1)}=K\breve{a}^{(K)}-K+1.\label{eq:Cor4i_a}
\end{equation}

ii) If $\breve{b}_{1}^{(K)}=\breve{b}_{2}^{(K)}=\ldots=\breve{b}_{K}^{(K)}\equiv\breve{b}^{(K)}$,
then
\begin{equation}
\breve{b}^{(1)}=K\breve{b}^{(K)}.\label{eq:Cor4ii_b}
\end{equation}

\end{cor}
With regard to the relationship between the shape parameters, in general,
(\ref{eq:Prop3_a}) reveals no evident monotonic dependency of $\breve{a}^{(1)}$
upon the number of mixture components. In the special case of the
component-wise equal $\breve{a}_{i}^{(K)}$'s, it is easily gathered
from (\ref{eq:Cor4i_a}) that $\breve{a}^{(1)}=K(\breve{a}^{(K)}-1)+1$,
which implies $\breve{a}^{(1)}$ may be constant in \emph{K} (under
$\breve{a}^{(K)}=1$), or increasing $(\breve{a}^{(K)}>1)$, or decreasing
$(\breve{a}^{(K)}<1)$\emph{K}. As far as the scale parameters are
concerned, they follow the pattern of the precisions entertained under
the precision-parametrized normal priors in Subsection \ref{sub:Normal}
(see (\ref{eq:Prop3_b}) and (\ref{eq:Prop1_v-precision})), rather
than the variances, which was the case under the inverse gamma priors. 

Contrary to the inverse gamma priors analyzed previously, working
under the gamma distributions provides an easy route to establishing
coherent values of the shape parameters once, again, models $M_{K}$
with various $K\in\mathbb{K}$ are at hand, and, given the number
of components, all the hyperparameters $\breve{a}_{i}^{(K)}$'s are
held equal. To this end, transform (\ref{eq:Cor4i_a}) and (\ref{eq:Cor4ii_b}),
respectively, into
\begin{equation}
\breve{a}^{(K)}=\frac{\breve{a}^{(1)}+K-1}{K}\label{eq:32}
\end{equation}
and
\begin{equation}
\breve{b}^{(K)}=\frac{1}{K}\breve{b}^{(1)}.\label{eq:33}
\end{equation}
Setting any $\breve{a}^{(1)}>0$ in (\ref{eq:32}) yields a positive
value of $\breve{a}^{(K)}$ for any $K\in\{2,3,\ldots\}$. Hence,
the approach disposes of \emph{a priori} fixing the maximum number
of components, otherwise necessitated under the inverse gamma framework.

\section{Coherence of constrained prior distributions\label{sec:Coherence-of-constrained}}

In the foregoing, only unconstrained priors, given by (\ref{eq:2})
and (\ref{eq:3}), under all the models have been considered. However,
in practice, it may be that some restrictions are to be imposed on
the parameters of the mixture model, usually aiming at ensuring the
identifiability of the mixture components or, possibly in addition
to that, some sort of regularity, such as the second-order stationarity
(in the time series framework). Therefore, in the present section,
we study the way in which introducing such parametric constraints
into the mixture model's prior affects the general results stated
in Lemma \ref{lem:Lemma1} and Corollary \ref{cor:Cor1}.

\subsection{Priors with identifiability constraints\label{sub:Priors-with-identifiability}}

There has already been a large variety of techniques advanced in the
literature to exert identifiability of the mixture model's components,
each procedure designed to tackle the widely-recognized label switching
issue, an inherent ailment of the mixture modeling. For a review and
more recent studies in the field we refer the reader to, e.g., \citet{JasraHolmesStephens:2005:LabelSwitching},\textbf{
}\citet{MarinMengersenRobert:2005:BayesianMixtures}, \citet{FruhwirthSchnatter:2006MixtureAndMS},
\citet{Yao:2012a:LabelSwitching1}, \citet{Yao:2012b:LabelSwitching2},
and the references therein. The most straightforward method (though
not universally recommended, according to the cited authors) consists
in imposing an inequality constraint upon the coordinates of the vector
$\lambda_{j}^{(K)}$ for a given $j\in\{1,2,\ldots,n\}$, such as
\begin{equation}
\lambda_{1,j}\leq\lambda_{2,j}\leq\ldots\leq\lambda_{K,j}.\label{eq:inequal_constraint}
\end{equation}
(Notice that the subscript \emph{j} is henceforth reintroduced in
the notation). We stress that it is the strict-inequalities variant
of (\ref{eq:inequal_constraint}) that is usually engaged in the literature,
thereby actually prohibiting the single-component structure from nesting
itself within the mixture model. Admittedly, such an approach is entirely
valid within the subjective setting, which, obviously, does not necessitate
establishing any relation between the models under consideration,
their priors included, even if such a one is conceivable. However,
aiming at ensuring the prior coherence between the single-component
model and the mixture model, with the latter's prior constrained,
does require allowing for the weak inequalities in (\ref{eq:inequal_constraint}),
for otherwise the former could not be obtained from the latter via
conditioning upon $\lambda_{1,j}=\lambda_{2,j}=\ldots=\lambda_{K,j}$,
$j=1,2,\ldots,n$. Finally, note that the distinction between the
weak and the strict inequalities within a continuous random variables
framework is hardly a matter of concern.

With no loss of generality we assume that the identifiability restriction
is imposed on the prior of $\lambda_{1}^{(K)}$, i.e., for $j=1$,
whereas the priors of the remaining $\lambda_{j}^{(K)}$'s $(j=2,3,\ldots,n)$
are unconstrained and coincide with (\ref{eq:4}). The prior for $\lambda_{1}^{(K)}$
can be written as
\begin{equation}
\pi_{\underline{\lambda_{1}^{(K)}}}(\lambda_{1}^{(K)}|M_{K})\propto\left[\overset{K}{\underset{i=1}{\prod}}\pi_{\underline{\lambda_{i,1}}}(\lambda_{i,1}|M_{K})\right]\mathbb{I}_{C_{K}}(\lambda_{1}^{(K)}),\label{eq:Constrained_prior_for_L1}
\end{equation}
where
\[
C_{K}=\{(c_{1}\; c_{2}\;\ldots\; c_{K})^{\prime}\in\mathbb{{R}}^{K}:c_{1}\leq c_{2}\leq\ldots\leq c_{K}\}
\]
and $\mathbb{I}_{C_{K}}(\cdot)$ represents the indicator function
of the set $C_{K}$. Incidentally, note a slight abuse of notation
in (\ref{eq:Constrained_prior_for_L1}), for $\pi_{\underline{\lambda_{i,1}}}(\lambda_{i,1}|M_{K})$
is actually no longer the marginal prior of $\lambda_{i,1}$, which
is due to the inequality constraint introducing stochastic dependency
between the coordinates of $\lambda_{1}^{(K)}$.

Proceeding along the same lines of reasoning as presented in Section
\ref{sec:Prior-coherence-in}, we rewrite (\ref{eq:Constrained_prior_for_L1})
under the reparametrized model, $\widetilde{M}_{K}$:
\begin{align}
\pi_{\underline{\widetilde{\lambda}_{1}^{(K)}}}(\widetilde{\lambda}_{1}^{(K)}|\widetilde{M}_{K}) & =\pi_{\underline{\lambda_{1}^{(K)}}}\left(g^{-1}(\widetilde{\lambda}_{1}^{(K)})|M_{K}\right)\nonumber \\
 & \propto\pi_{\underline{\lambda_{1,1}}}\left(\lambda_{1,1}|M_{K}\right)\left[\overset{K}{\underset{i=2}{\prod}}\pi_{\underline{\lambda_{i,1}}}\left(\tau_{i,1}+\lambda_{1,1}|M_{K}\right)\right]\mathbb{I}_{C_{K-1}^{+}}(\tau_{1}),\label{eq:39}
\end{align}
where
\[
C_{K-1}^{+}=\{(c_{1}\; c_{2}\;\ldots\; c_{K-1})^{\prime}\in\mathbb{{R}}^{K-1}:0\leq c_{1}\leq c_{2}\leq\ldots\leq c_{K-1}\}
\]
and $\tau_{1}=(\tau_{2,1}\;\tau_{3,1}\;\ldots\;\tau_{K,1})^{\prime}$,
$\tau_{i,1}=\lambda_{i,1}-\lambda_{1,1}$ $(i=2,3,...,K)$, so that
the presence of $\mathbb{I}_{C_{K-1}^{+}}(\tau_{1})$ in (\ref{eq:39})
is equivalent to restricting the contrasts with the inequality $0\leq\tau_{2,1}\leq\tau_{3,1}\leq\ldots\leq\tau_{K,1}$.
Now, recognizing that $\mathbb{I}_{C_{K-1}^{+}}(0_{[(K-1)\times1]})=1$,
and following the proof of Lemma \ref{lem:Lemma1} we obtain:
\begin{align*}
\pi_{\underline{\lambda_{1,1}}}(\lambda_{1,1}|M_{1}) & =\pi_{\underline{\lambda_{1,1}}|\underline{\tau_{1}}}(\lambda_{1,1}|\tau_{1}=0_{[(K-1)\times1]},\widetilde{M}_{K})\\
 & \propto\pi_{\underline{\lambda_{1,1}},\underline{\tau_{1}}}(\lambda_{1,j},\tau_{1}=0_{[(K-1)\times1]}|\widetilde{M}_{K})\\
 & \propto\pi_{\underline{\lambda_{1,1}}}(\lambda_{1,1}|M_{K})\left[\overset{K}{\underset{i=2}{\prod}}\pi_{\underline{\lambda_{i,1}}}(\lambda_{1,1}|M_{K})\right]\mathbb{I}_{C_{K-1}^{+}}(0_{[(K-1)\times1]})\\
 & =\overset{K}{\underset{i=1}{\prod}}\pi_{\underline{\lambda_{i,1}}}(\lambda_{1,1}|M_{K}),
\end{align*}
which coincides with the result displayed in the lemma. Hence, we
conclude that constraining the mixture model's prior with an identifiability
restriction does not affect the coherence conditions stated in Lemma
\ref{lem:Lemma1} and Corollary \ref{cor:Cor1}.

\subsection{Priors with regularity constraints\label{sub:Priors-with-regularity}}

Another common type of parametric restrictions introduced into statistical
models are the ones enforcing some sort of regularity, arising from
the theory underlying the phenomenon at hand or being of a rather
technical nature (e.g., ensuring the second-order stationarity in
the time series framework). Therefore, we move on to establishing
the way in which a regularity restriction imposed upon the mixture
model's prior translates into the form of the coherent prior under
the single-component specification.

Let $\zeta_{K}(\cdot):\Theta^{(K)}\rightarrow\mathbb{R}$ be such
a function of $\theta^{(K)}$ that the regularity constraint under
$M_{K}$ is satisified if and only if $\zeta_{K}(\theta^{(K)})\in R_{K}\subset\mathbb{R}$
(or, equivalently, $\mathbb{{I}}_{R_{K}}\left\{ \zeta_{K}(\theta^{(K)})\right\} =1$),
and $\zeta_{K}(\theta^{(K)})$ becomes invariant with respect to $\eta$
under the reducing restrictions given by (\ref{eq:8}).

Rewriting $\theta^{(K)}=(\delta^{\prime}\;\lambda_{1}^{(K)\prime}\;\lambda_{2}^{(K)\prime}\;\ldots\;\lambda_{n}^{(K)\prime}\;\eta^{\prime})^{\prime}$
as $\theta^{(K)}=(\delta^{\prime}\;\lambda^{(K)\prime}\;\eta^{\prime})^{\prime}$
with $\lambda^{(K)}=(\lambda_{1}^{(K)\prime}\;\lambda_{2}^{(K)\prime}\;\ldots\;\lambda_{n}^{(K)\prime})^{\prime}$,
and assuming prior independence (though only up to the regularity
restriction), the constrained prior under the mixture model presents
itself as
\begin{align*}
\pi_{\underline{\theta^{(K)}}}(\theta^{(K)}|M_{K}) & \propto\pi_{\underline{\delta}}(\delta|M_{K})\pi_{\underline{\lambda^{(K)}}}(\lambda^{(K)}|M_{K})\pi_{\underline{\eta}}(\eta|M_{K})\mathbb{{I}}_{R_{K}}\left\{ \zeta_{K}(\theta^{(K)})\right\} \\
 & =\pi_{\underline{\delta}}(\delta|M_{K})\left[\overset{n}{\underset{j=1}{\prod}}\pi_{\underline{\lambda_{j}^{(K)}}}(\lambda_{j}^{(K)}|M_{K})\right]\pi_{\underline{\eta}}(\eta|M_{K})\mathbb{{I}}_{R_{K}}\left\{ \zeta_{K}(\theta^{(K)})\right\} ,
\end{align*}
with each $\pi_{\underline{\lambda_{j}^{(K)}}}(\lambda_{j}^{(K)}|M_{K})$
$(j=1,2,\ldots,n)$ being given by (\ref{eq:4}). As regards deriving
from the above expression the coherent prior under the single-component
model, one conjectures that it is also to be constrained with some
restriction, say $\zeta_{1}(\theta^{(1)})\in R_{1}\subset\mathbb{R}$.
Although a precise relation between $\zeta_{1}(\cdot):\Theta^{(1)}\rightarrow\mathbb{R}$
and $\zeta_{K}(\cdot)$ is yet to be specified, we shall write a prototypical
form, so to say, of the prior under $M_{1}$:
\begin{align}
\pi_{\underline{\theta^{(1)}}}(\theta^{(1)}|M_{1}) & \propto\pi_{\underline{\delta}}(\delta|M_{1})\left[\overset{n}{\underset{j=1}{\prod}}\pi_{\underline{\lambda_{1,j}}}(\lambda_{1,j}|M_{1})\right]\mathbb{{I}}_{R_{1}}\left\{ \zeta_{1}(\theta^{(1)})\right\} .\label{eq:Prior_L1_prototype}
\end{align}

Further, let us recast $M_{K}$ into $\widetilde{M}_{K}$ (with the
transform \emph{g} affecting only $\lambda_{j}^{(K)}$'s, as in Section
\ref{sec:Prior-coherence-in}), so that
\begin{align*}
\pi_{\underline{\widetilde{\theta}^{(K)}}}(\widetilde{\theta}^{(K)}|\widetilde{M}_{K}) & \propto\pi_{\underline{\delta}}(\delta|\widetilde{M}_{K})\pi_{\underline{\widetilde{\lambda}^{(K)}}}(\widetilde{\lambda}^{(K)}|\widetilde{M}_{K})\pi_{\underline{\eta}}(\eta|\widetilde{M}_{K})\\
 & \times\mathbb{{I}}_{R_{K}}\left\{ \zeta_{K}\left(\delta,g^{-1}(\widetilde{\lambda}_{1}^{(K)}),g^{-1}(\widetilde{\lambda}_{2}^{(K)}),\ldots,g^{-1}(\widetilde{\lambda}_{n}^{(K)}),\eta\right)\right\} ,
\end{align*}
where $\widetilde{\lambda}^{(K)}=(\widetilde{\lambda}_{1}^{(K)\prime}\;\widetilde{\lambda}_{2}^{(K)\prime}\;\ldots\;\widetilde{\lambda}_{n}^{(K)\prime})^{\prime}$,
$\pi_{\underline{\delta}}(\delta|\widetilde{M}_{K})=\pi_{\underline{\delta}}(\delta|M_{K})$,
$\pi_{\underline{\eta}}(\eta|\widetilde{M}_{K})=\pi_{\underline{\eta}}(\eta|M_{K})$
and
\[
\pi_{\underline{\widetilde{\lambda}^{(K)}}}(\widetilde{\lambda}^{(K)}|\widetilde{M}_{K})=\overset{n}{\underset{j=1}{\prod}}\pi_{\underline{\widetilde{\lambda}_{j}^{(K)}}}(\widetilde{\lambda}_{j}^{(K)}|\widetilde{M}_{K}).
\]
Employing the end result of (\ref{eq:7}) into $\pi_{\underline{\widetilde{\theta}^{(K)}}}(\widetilde{\theta}^{(K)}|\widetilde{M}_{K})$,
one obtains
\begin{align*}
\pi_{\underline{\widetilde{\theta}^{(K)}}}(\widetilde{\theta}^{(K)}|\widetilde{M}_{K}) & \propto\pi_{\underline{\delta}}(\delta|M_{K})\left[\overset{n}{\underset{j=1}{\prod}}\left(\pi_{\underline{\lambda_{1,j}}}\left(\lambda_{1,j}|M_{K}\right)\overset{K}{\underset{i=2}{\prod}}\pi_{\underline{\lambda_{i,j}}}\left(\tau_{i,j}+\lambda_{1,j}|M_{K}\right)\right)\right]\\
 & \times\pi_{\underline{\eta}}(\eta|M_{K})\mathbb{{I}}_{R_{K}}\left\{ \zeta_{K}\left(\delta,g^{-1}(\widetilde{\lambda}_{1}^{(K)}),g^{-1}(\widetilde{\lambda}_{2}^{(K)}),\ldots,g^{-1}(\widetilde{\lambda}_{n}^{(K)}),\eta\right)\right\} .
\end{align*}
Now, notice that under $\tau_{j}=0_{[(K-1)\times1]}$, we get $g^{-1}(\widetilde{\lambda}_{j}^{(K)})=\lambda_{1,j}\iota_{K}$,
with $\iota_{K}=(1\;1\;\ldots\;1)^{\prime}\in\mathbb{R}^{K}$ and
$j=1,2,\ldots,n$. Finally, the coherent prior distribution under
$M_{1}$ is derived:
\begin{align*}
\pi_{\underline{\theta^{(1)}}}(\theta^{(1)}|M_{1}) & \propto\pi_{\underline{\delta}}(\delta|M_{K})\\
 & \times\pi_{\underline{\widetilde{\lambda}^{(K)}}|\underline{\tau_{2}},\underline{\tau_{3}},\ldots,\underline{\tau_{n}}}(\widetilde{\lambda}^{(K)}|\tau_{2}=\tau_{3}=\ldots=\tau_{n}=0_{[(K-1)\times1]},\widetilde{M}_{K})\\
 & \times\mathbb{{I}}_{R_{K}}\left\{ \zeta_{K}\left(\delta,\lambda_{1,1}\iota_{K},\lambda_{1,2}\iota_{K},\ldots,\lambda_{1,n}\iota_{K},\eta\right)\right\} \\
 & =\pi_{\underline{\delta}}(\delta|M_{K})\left[\overset{n}{\underset{j=1}{\prod}}\overset{K}{\underset{i=1}{\prod}}\pi_{\underline{\lambda_{i,j}}}\left(\lambda_{1,j}|M_{K}\right)\right]\\
 & \times\mathbb{{I}}_{R_{K}}\left\{ \zeta_{K}\left(\delta,\lambda_{1,1}\iota_{K},\lambda_{1,2}\iota_{K},\ldots,\lambda_{1,n}\iota_{K},\eta\right)\right\} .
\end{align*}
To reconcile the above expression with (\ref{eq:Prior_L1_prototype}),
the following conditions must hold simultaneously:
\begin{equation}
\pi_{\underline{\delta}}(\delta|M_{1})=\pi_{\underline{\delta}}(\delta|M_{K}),\label{eq:Regularity_cond_1}
\end{equation}

\begin{equation}
\pi_{\underline{\lambda_{1,j}}}(\lambda_{1,j}|M_{1})\propto\overset{K}{\underset{i=1}{\prod}}\pi_{\underline{\lambda_{i,j}}}\left(\lambda_{1,j}|M_{K}\right),\label{eq:Regularity_cond_2}
\end{equation}

\begin{equation}
\mathbb{{I}}_{R_{1}}\left\{ \zeta_{1}(\theta^{(1)})\right\} =1\Leftrightarrow\mathbb{{I}}_{R_{K}}\left\{ \zeta_{K}\left(\delta,\lambda_{1,1}\iota_{K},\lambda_{1,2}\iota_{K},\ldots,\lambda_{1,n}\iota_{K},\eta\right)\right\} =1.\label{eq:Regularity_cond_3}
\end{equation}
Note that (\ref{eq:Regularity_cond_1}) and (\ref{eq:Regularity_cond_2})
coincide with (\ref{eq:5}) and (\ref{eq:Lemma1_1}), respectively.
Hence, from (\ref{eq:Regularity_cond_1})-(\ref{eq:Regularity_cond_3})
it follows that in order to design a coherent prior under the single-component
model one needs to:
\begin{enumerate}
\item Comply with the rules formulated for the case of unconstrained priors;
see (\ref{eq:5}) and Lemma \ref{lem:Lemma1}.
\item Restrain the single-component model's prior with a restriction equivalent
to the one restraining the mixture model's prior under the nesting
restrictions, given by (\ref{eq:8}).
\end{enumerate}

\section{Example: A coherent prior structure for a class of stationary Markov-switching
AR(2) models\label{sec:Example}}

Consider the following \emph{K}-state Markov-switching AR(2) model:
\begin{equation}
y_{t}=\alpha_{S_{t}}+\phi_{S_{t},1}y_{t-1}+\phi_{S_{t},2}y_{t-2}+\sigma_{S_{t}}\varepsilon_{t},\label{eq:MSIAH_AR2_model}
\end{equation}
where $\varepsilon_{t}\sim iiN(0,1)$ and the sequence $\{S_{t}\}$
forms a homogeneous and ergodic Markov chain with finite state-space
$\mathbb{{S}}=\{1,2,\ldots,K\}$ and transition probabilities $\eta_{ij}\equiv\Pr{(S_{t}=j|S_{t-1}=i)}$,
arrayed in transition matrix $P=[\eta_{ij}]_{i,j=1,2,\ldots,K}$.
Adopting the convention introduced by \citet{Krolzig:1997MSVAR},
we refer to (\ref{eq:MSIAH_AR2_model}) as the MSIAH(\emph{K})-AR(2)
model (or, $M_{K}$, in short), which indicates allowing all the parameters
to feature Markovian breaks, i.e., the intercept, the autoregressive
coefficients and the error term's variance. Let $\alpha^{(K)}=(\alpha_{1}\;\alpha_{2}\;\ldots\;\alpha_{K})^{\prime}$,
$\phi_{1}^{(K)}=(\phi_{1,1}\;\phi_{2,1}\;\ldots\;\phi_{K,1})^{\prime}$,
$\phi_{2}^{(K)}=(\phi_{1,2}\;\phi_{2,2}\;\ldots\;\phi_{K,2})^{\prime}$,
$\varsigma^{(K)}=(\sigma_{1}^{-2}\;\sigma_{2}^{-2}\;\ldots\;\sigma_{K}^{-2})^{\prime}$,
and $\eta$ be structured as described in Section \ref{sec:Prior-coherence-in},
so that
\[
\theta^{(K)}=(\alpha^{(K)\prime}\;\phi_{1}^{(K)\prime}\;\phi_{2}^{(K)\prime}\;\varsigma^{(K)\prime}\;\eta^{\prime})^{\prime}.
\]
The model under consideration generalizes the following AR(2) specification
(hereafter denoted by $M_{1}$):
\begin{equation}
y_{t}=\alpha+\phi_{1}y_{t-1}+\phi_{2}y_{t-2}+\sigma\varepsilon_{t}\label{eq:AR2_model}
\end{equation}
in that $M_{K}$ introduces discrete changes into each of the four
parameters of $M_{1}$ (grouped in $\theta^{(1)}=(\alpha\;\phi_{1}\;\phi_{2}\;\sigma^{-2})^{\prime}$).

Based on the results provided by \citet{FrancqZakoian:2001:MSVAR},
for the MSIAH(\emph{K})-AR(2) process to be nonanticipative (i.e.,
causal) and second-order stationary it suffices that
\begin{equation}
\rho(P_{2})<1,\label{eq:Stationarity_cond_MSIAH_AR2}
\end{equation}
where $\rho(P_{2})$ signifies the spectral radius of matrix $P_{2}$
defined as
\begin{equation}
P_{2}=\begin{pmatrix}\eta_{11}(\Phi_{1}\otimes\Phi_{1}) & \eta_{21}(\Phi_{1}\otimes\Phi_{1}) & \cdots & \eta_{K1}(\Phi_{1}\otimes\Phi_{1})\\
\eta_{12}(\Phi_{2}\otimes\Phi_{2}) & \eta_{22}(\Phi_{2}\otimes\Phi_{2}) & \cdots & \eta_{K2}(\Phi_{2}\otimes\Phi_{2})\\
\vdots & \vdots &  & \vdots\\
\eta_{1K}(\Phi_{K}\otimes\Phi_{K}) & \eta_{2K}(\Phi_{K}\otimes\Phi_{K}) & \cdots & \eta_{KK}(\Phi_{K}\otimes\Phi_{K})
\end{pmatrix},\label{eq:P2_matrix}
\end{equation}
with
\[
\Phi_{k}=\begin{pmatrix}\phi_{k,1} & \phi_{k,2}\\
1 & 0
\end{pmatrix},\qquad k=1,2,\ldots,K,
\]
and $\otimes$ denoting the matrix tensor product. Assuming the mutual
independence of $\theta^{(K)}$'s individual components, the prior
under $M_{K}$ can be written as
\begin{align*}
\pi(\theta^{(K)}|M_{K}) & =\pi(\alpha^{(K)}|M_{K})\pi(\phi_{1}^{(K)}|M_{K})\pi(\phi_{2}^{(K)}|M_{K})\\
 & \times\pi(\varsigma^{(K)}|M_{K})\pi(\eta|M_{K})\mathbb{{I}}_{R_{K}}\{\rho(P_{2})\},
\end{align*}
where $R_{K}=[0,1)$. Note that we simplified the notation by dropping
the subscripts indexing densities, and write, generally, $\pi(\omega)$
instead of $\pi_{\underline{\omega}}(\omega)$. Similarly, the prior
under $M_{1}$ is given by 
\begin{align*}
\pi(\theta^{(1)}|M_{1}) & =\pi(\alpha|M_{1})\pi(\phi_{1}|M_{1})\pi(\phi_{2}|M_{1})\\
 & \times\pi(\sigma^{-2}|M_{1})\mathbb{{I}}_{R_{1}}\{\zeta_{1}(\theta^{(1)})\}.
\end{align*}
Notice that, for the sake of exposition, we do not impose any identifiability
restriction upon $\pi(\theta^{(K)}|M_{K})$, though we stress that
it would not alter the following considerations (see Subsection \ref{sub:Priors-with-identifiability}).

To derive the specific forms of $\zeta_{1}(\theta^{(1)})$ and $R_{1}$,
complying with the coherence condition given by (\ref{eq:Regularity_cond_3}),
one needs to ponder (\ref{eq:Stationarity_cond_MSIAH_AR2}) under
the equality restrictions: $\phi_{1,1}=\phi_{2,1}=\ldots=\phi_{K,1}\equiv\phi_{1}$
and $\phi_{1,2}=\phi_{2,2}=\ldots=\phi_{K,2}\equiv\phi_{2}$. (Notice
that the switching intercepts, $\alpha^{(K)}$, and the error term's
precisions, $\varsigma^{(K)}$, do not need to be restricted with
the nesting equalities, in the process). With that provided, the matrices
$\Phi_{k}$'s collapse into
\[
\Phi=\begin{pmatrix}\phi_{1} & \phi_{2}\\
1 & 0
\end{pmatrix},
\]
which coincides with the companion matrix for the AR(2) process defined
in (\ref{eq:AR2_model}). Supplanting $\Phi_{k}$'s with $\Phi$ in
(\ref{eq:P2_matrix}) we obtain
\[
P_{2}=\begin{pmatrix}\eta_{11}(\Phi\otimes\Phi) & \eta_{21}(\Phi\otimes\Phi) & \cdots & \eta_{K1}(\Phi\otimes\Phi)\\
\eta_{12}(\Phi\otimes\Phi) & \eta_{22}(\Phi\otimes\Phi) & \cdots & \eta_{K2}(\Phi\otimes\Phi)\\
\vdots & \vdots &  & \vdots\\
\eta_{1K}(\Phi\otimes\Phi) & \eta_{2K}(\Phi\otimes\Phi) & \cdots & \eta_{KK}(\Phi\otimes\Phi)
\end{pmatrix}=P^{\prime}\otimes\Phi\otimes\Phi.
\]
Then $\rho(P_{2})=\rho(P^{\prime}\otimes\Phi\otimes\Phi)=\rho(P^{\prime})[\rho(\Phi)]^{2}=[\rho(\Phi)]^{2}$,
for $P$ is a stochastic matrix. Finally,
\begin{align*}
\mathbb{{I}}_{R_{K}}\{\rho(P_{2})\}=1 & \Leftrightarrow\mathbb{{I}}_{R_{K}}\left\{ [\rho(\Phi)]^{2}\right\} =1\\
 & \Leftrightarrow\mathbb{{I}}_{R_{K}}\left\{ \rho(\Phi)\right\} =1.
\end{align*}
The latter expression requires that the maximum absolute eigenvalue
of $\Phi$ be less than one, which is equivalent to the well-known
condition for the AR(2) process to be nonanticipative and second-order
stationary, necessitating all eigenvalues of the companion matrix
to fall within the interval $(-1,1)$. Therefore, we assume that $\zeta_{1}(\theta^{(1)}):=\rho(\Phi)$
and $R_{1}=R_{K}=[0,1)$.

As regards particular choice for the individual densities comprising
$\pi(\theta^{(K)}|M_{K})$, while keeping to the assumptions stated
in Section \ref{sec:Prior-coherence-in}, we follow a typical framework
by setting
\begin{itemize}
\item normal distributions for the coordinates of $\alpha^{(K)}$, $\phi_{1}^{(K)}$
and $\phi_{2}^{(K)}$:
\begin{equation}
\pi(\alpha^{(K)}|M_{K})=\overset{K}{\underset{i=1}{\prod}}f_{N}\left(\alpha_{i}|m_{\alpha}^{(K)},(\breve{v}_{\alpha}^{(K)})^{-1}\right),\label{eq:Prior_alfa_MK}
\end{equation}
\begin{equation}
\pi(\phi_{1}^{(K)}|M_{K})=\overset{K}{\underset{i=1}{\prod}}f_{N}\left(\phi_{i,1}|m_{\phi_{1}}^{(K)},(\breve{v}_{\phi_{1}}^{(K)})^{-1}\right),\label{eq:Prior_fi1_MK}
\end{equation}
\begin{equation}
\pi(\phi_{2}^{(K)}|M_{K})=\overset{K}{\underset{i=1}{\prod}}f_{N}\left(\phi_{i,2}|m_{\phi_{2}}^{(K)},(\breve{v}_{\phi_{2}}^{(K)})^{-1}\right);\label{eq:Prior_fi2_MK}
\end{equation}

\item gamma distributions for the coordinates of $\varsigma^{(K)}$ (or,
alternatively, the inverse gamma distributions for the variances $\sigma_{i}^{2}$,
$i=1,2,\ldots,K$):
\begin{equation}
\pi(\varsigma^{(K)}|M_{K})=\overset{K}{\underset{i=1}{\prod}}f_{G}\left(\sigma_{i}^{-2}|\breve{a}^{(K)},\breve{b}^{(K)}\right);\label{eq:Prior_sigma_MK}
\end{equation}

\item Dirichlet distributions for the (\emph{a priori }independent) rows
of the transition matrix:
\begin{equation}
\pi(\eta_{1},\eta_{2},\ldots,\eta_{K}|M_{K})=\overset{K}{\underset{i=1}{\prod}}f_{Dir}\left(\eta_{i}|d_{i}^{(K)}\right),\label{eq:Prior_eta_MK}
\end{equation}
with $d_{i}^{(K)}=(d_{i,1}\; d_{i,2}\;\ldots\; d_{i,K})^{\prime}$
standing for the vector of the hyperparameters.
\end{itemize}
Note that, quite customarily, equal hyperparameters over the regimes
are assumed in (\ref{eq:Prior_alfa_MK})-(\ref{eq:Prior_sigma_MK}).

Following the results presented in Propositions \ref{prop:Prop1}
and \ref{prop:Prop3}, coherent priors under $M_{1}$ can be written
as
\begin{equation}
\pi(\alpha|M_{1})=f_{N}\left(\alpha|m_{\alpha}^{(1)},(\breve{v}_{\alpha}^{(1)})^{-1}\right),\label{eq:Prior_alfa_M1}
\end{equation}

\begin{equation}
\pi(\phi_{1}|M_{1})=f_{N}\left(\phi_{1}|m_{\phi_{1}}^{(1)},(\breve{v}_{\phi_{1}}^{(1)})^{-1}\right),\label{eq:Prior_fi1_M1}
\end{equation}

\begin{equation}
\pi(\phi_{2}|M_{1})=f_{N}\left(\phi_{2}|m_{\phi_{2}}^{(1)},(\breve{v}_{\phi_{2}}^{(1)})^{-1}\right),\label{eq:Prior_fi2_M1}
\end{equation}

\begin{equation}
\pi(\sigma^{-2}|M_{1})=f_{G}\left(\sigma^{-2}|\breve{a}^{(1)},\breve{b}^{(1)}\right),\label{eq:Prior_sigma_M1}
\end{equation}
with the hyperparameters related with the ones displayed in (\ref{eq:Prior_alfa_MK})-(\ref{eq:Prior_sigma_MK})
via Formulae (\ref{eq:Cor2ii_m}) and (\ref{eq:Cor2ii_v-precision})
(in the case of the normals), and (\ref{eq:Cor4i_a}) and (\ref{eq:Cor4ii_b})
(in the case of the gamma distributions).

Since the hyperparameters for each group of the switching parameters
under $M_{K}$ are held equal across the regimes, there are actually
two routes available to establish ceherent prior structures. Within
the first one, one sets the values of the hyperparameters under the
general model first, and then the ones under the single-component
model. Within the second approach, one proceeds the other way round.
However, should different vaules of the hyperparameters for a given
group of the switching parameters under $M_{K}$ be allowed, then
only the first of the two strategies can be followed, with the relevant
formulae provided in Propositions \ref{prop:Prop1} and \ref{prop:Prop3}.

Eventually, notice that the two: the AR(2) and the MSIAH(\emph{K})-AR(2)
model, represent the extremes, with the former featuring no switches
at all, and the latter, on the other hand, introducting Markovian
breaks into all the four coefficients at once: the intercept, the
two autoregressive parameters, and the error term's variance. Therefore,
the two specifications do not share any common parameters. Naturally,
one may be prompted to limit the set of the parameters enabled to
switch to include only one, two, or three out of the four, in each
case obtaining some ``intermediate'' specification. Should that
be the case, our methodology for establishing coherent priors applies
straightforwardly. To deliver some illustrative example, consider
an AR(2) model with switches introduced only into the intercept, hereafter
denoted as MSI(\emph{K})-AR(2) or $M_{K}^{*}$, in short. Obviously,
it forms one of all the\textbf{ }conceivable ``intermediate'' specifications,
nesting the single-component AR(2) model on the one hand, and being
nested within the MSIAH(\emph{K})-AR(2) model, on the other. Write
$\theta_{*}^{(K)}=(\alpha^{(K)\prime}\;\phi_{1}\;\phi_{2}\;\sigma^{-2}\;\eta^{\prime})^{\prime}$
for the vector of $M_{K}^{*}$'s parameters, with $\alpha^{(K)}=(\alpha_{1}\;\alpha_{2}\;\ldots\;\alpha_{K})^{\prime}$.
The prior is structured as
\begin{align*}
\pi(\theta_{*}^{(K)}|M_{K}^{*}) & =\pi(\alpha^{(K)}|M_{K}^{*})\pi(\phi_{1}|M_{K}^{*})\pi(\phi_{2}|M_{K}^{*})\\
 & \times\pi(\sigma{}^{-2}|M_{K}^{*})\pi(\eta|M_{K}^{*})\mathbb{{I}}_{R_{K}^{*}}\{\zeta_{K}^{*}(\theta_{*}^{(K)})\},
\end{align*}
where, according to the argumentation presented above, the regularity
restriction assumes the form of the one derived for the single-component
model: $\zeta_{K}^{*}(\theta_{*}^{(K)}):=\rho(\Phi)$ and $R_{K}^{*}=R_{1}=[0,1)$.
Assuming equal hyperparameters for $\alpha^{(K)}$'s prior, in order
to establish such a prior structure under $M_{K}^{*}$ that is coherent
with that of $M_{1}$ we set (\ref{eq:Prior_alfa_MK}) for $\pi(\alpha^{(K)}|M_{K}^{*})$,
and (\ref{eq:Prior_fi1_M1})-(\ref{eq:Prior_sigma_M1}) for $\pi(\phi_{1}|M_{K}^{*})$,
$\pi(\phi_{2}|M_{K}^{*})$ and $\pi(\sigma^{-2}|M_{K}^{*})$, respectively.
Notice that if, in addition to that, the density $\pi(\eta|M_{K}^{*})$
coincides with (\ref{eq:Prior_eta_MK}), then the prior structure
of $M_{K}^{*}$ is also coherent with the one specified under the
general model, $M_{K}$.

\bibliographystyle{chicago}
\bibliography{references}

\subsubsection*{Acknowledgements}

The research was realized within a project financed by the National
Science Center (Poland) under decision No. DEC-2011/01/N/HS4/03105.\pagebreak{}

\appendix

\section{Appendix: Proof of Proposition 1}

Invoking Lemma \ref{lem:Lemma1} and performing some simple manipulations,
the proof proceeds as follows:
\begin{align*}
\pi_{\underline{\lambda_{1}}}(\lambda_{1}|M_{1}) & \propto\overset{K}{\underset{i=1}{\prod}}\pi_{\underline{\lambda_{i}}}(\lambda_{1}|M_{K})=\overset{K}{\underset{i=1}{\prod}}f_{N}^{(1)}(\lambda_{1}|m_{i}^{(K)},v_{i}^{(K)})\\
 & \propto\exp\left\{ -\frac{1}{2}\overset{K}{\underset{i=1}{\sum}}\frac{(\lambda_{1}-m_{i}^{(K)})^{2}}{v_{i}^{(K)}}\right\} \\
 & \propto\exp\left\{ -\frac{1}{2}\overset{K}{\underset{i=1}{\sum}}\left(\frac{1}{v_{i}^{(K)}}\lambda_{1}^{2}-2\frac{m_{i}^{(K)}}{v_{i}^{(K)}}\lambda_{1}\right)\right\} \\
 & \propto\exp\left\{ -\frac{1}{2}\left(\overset{K}{\underset{i=1}{\sum}}\frac{1}{v_{i}^{(K)}}\right)\left(\lambda_{1}^{2}-2\lambda_{1}\frac{\overset{K}{\underset{i=1}{\sum}}\frac{m_{i}^{(K)}}{v_{i}^{(K)}}}{\overset{K}{\underset{i=1}{\sum}}\frac{1}{v_{i}^{(K)}}}\right)\right\} \\
 & \propto\exp\left\{ -\frac{1}{2\left(\overset{K}{\underset{i=1}{\sum}}\frac{1}{v_{i}^{(K)}}\right)^{-1}}\left(\lambda_{1}-\frac{\overset{K}{\underset{i=1}{\sum}}\frac{m_{i}^{(K)}}{v_{i}^{(K)}}}{\overset{K}{\underset{i=1}{\sum}}\frac{1}{v_{i}^{(K)}}}\right)^{2}\right\} \\
 & \propto f_{N}^{(1)}(\lambda_{1}|m^{(1)},v^{(1)}),
\end{align*}
with $m^{(1)}$ and $v^{(1)}$ given by (\ref{eq:Prop1_m-variance})
and (\ref{eq:Prop1_v-variance}), respectively. The proof for the
precision-parametrized normal densities follows analogously.

\pagebreak{}

\section{Appendix: Proof of Proposition 2}

The proof is analogous to the one presented for Proposition \ref{prop:Prop1}:
\begin{align*}
\pi_{\underline{\lambda_{1}}}(\lambda_{1}|M_{1}) & \propto\overset{K}{\underset{i=1}{\prod}}\pi_{\underline{\lambda_{i}}}(\lambda_{1}|M_{K})=\overset{K}{\underset{i=1}{\prod}}f_{IG}(\lambda_{1}|a_{i}^{(K)},b_{i}^{(K)})\\
 & \propto\left[\overset{K}{\underset{i=1}{\prod}}(\lambda_{1})^{-(a_{i}^{(K)}+1)}\right]\exp\left\{ -\frac{1}{\lambda_{1}}\overset{K}{\underset{i=1}{\sum}}\frac{1}{b_{i}^{(K)}}\right\} \\
 & =(\lambda_{1})^{-\left(\overset{K}{\underset{i=1}{\sum}}a_{i}^{(K)}+K-1+1\right)}\exp\left\{ -1\left/\lambda_{1}\left(\overset{K}{\underset{i=1}{\sum}}\frac{1}{b_{i}^{(K)}}\right)^{-1}\right.\right\} \\
 & \propto f_{IG}(\lambda_{1}|a^{(1)},b^{(1)}),
\end{align*}
with $a^{(1)}$ and $b^{(1)}$ given by (\ref{eq:Prop2_a}) and (\ref{eq:Prop2_b}),
respectively.

\pagebreak{}

\section{Appendix: Proof of Proposition 3}

We proceed analogously to the proofs of Propostions \ref{prop:Prop1}
and \ref{prop:Prop2}:
\begin{align*}
\pi_{\underline{\lambda_{1}}}(\lambda_{1}|M_{1}) & \propto\overset{K}{\underset{i=1}{\prod}}\pi_{\underline{\lambda_{i}}}(\lambda_{1}|M_{K})=\overset{K}{\underset{i=1}{\prod}}f_{G}(\lambda_{1}|\breve{a}_{i}^{(K)},\breve{b}_{i}^{(K)})\\
 & \propto\left[\overset{K}{\underset{i=1}{\prod}}(\lambda_{1})^{\breve{a}_{i}^{(K)}-1}\right]\exp\left\{ -\lambda_{1}\overset{K}{\underset{i=1}{\sum}}\breve{b}_{i}^{(K)}\right\} \\
 & =(\lambda_{1})^{\overset{K}{\underset{i=1}{\sum}}\breve{a}_{i}^{(K)}-K+1-1}\exp\left\{ -\lambda_{1}\overset{K}{\underset{i=1}{\sum}}\breve{b}_{i}^{(K)}\right\} \\
 & \propto f_{G}(\lambda_{1}|\breve{a}^{(1)},\breve{b}^{(1)}),
\end{align*}
with $\breve{a}^{(1)}$ and $\breve{b}^{(1)}$ given by (\ref{eq:Prop3_a})
and (\ref{eq:Prop3_b}), respectively.
\end{document}